\documentclass{amsproc}
\usepackage{amssymb}
\usepackage{amsmath}
\usepackage{amsfonts}

\theoremstyle{plain}

\newtheorem{lemma}{Lemma}
\newtheorem{notation}{Notation}

\newtheorem{theorem}{Theorem}
\numberwithin{equation}{section}
\input{tcilatex}

\begin{document}
\title[Dodgson's Algorithm]{A Mathematical Proof of Dodgson's Algorithm}
\author{Kouachi Said}
\address{ University Centre of Khenchela 40100 Algeria.\\
}
\email{kouachi.said@caramail.com}
\urladdr{}
\thanks{}
\author{Abdelmalek Salem}
\curraddr{Department of mathematics, University Centre of Tebessa 12002
Algeria.}
\email{a.salem@gawab.com}
\urladdr{}
\thanks{}
\author{Rebiai Belgacem}
\address{Department of mathematics, University Centre of Tebessa 12002
Algeria.}
\email{brebia\"{\i}@gmail.com}
\urladdr{}
\date{2006}
\dedicatory{}
\thanks{}

\begin{abstract}
In this paper we give a mathematical proof of Dodgson algorithm \cite%
{Dodgson}. Recently Zeilberger \cite{Zeilberger} gave a bijective proof. Our
techniques are based on determinant properties and they are obtained by
induction.
\end{abstract}

\maketitle

\section{Introduction}

To prove mathematically the well-known Dodgson's algorithm, concerning a
square matrix $A=\left( a_{i,j}\right) _{1\leq i,j\leq n}$ we may generalize
it as follows:

\begin{equation}
\left. 
\begin{array}{c}
\det \left[ \left( a_{i,j}\right) _{1\leq i,j\leq n}\right] \det \left[
\left( a_{i,j}\right) _{\substack{ i\neq k,l  \\ j\neq k,l}}\right] = \\ 
\det \left[ \left( a_{i,j}\right) _{\substack{ i\neq l  \\ j\neq l}}\right]
\det \left[ \left( a_{i,j}\right) _{\substack{ i\neq k  \\ j\neq k}}\right]
-\det \left[ \left( a_{i,j}\right) _{\substack{ i\neq l  \\ j\neq k}}\right]
\det \left[ \left( a_{i,j}\right) _{\substack{ i\neq k  \\ j\neq l}}\right] ,%
\end{array}%
\right.  \label{(1.1)}
\end{equation}%
for all $k,l=1,...,n$ considering $k<l.$

For this purpose, we need some notations

\subsection{Notations}

\begin{notation}
The $\left( n-k\right) \times \left( n-l\right) $ matrix obtained from $A$
by removing the $i_{1}^{th},i_{2}^{th}...i_{k}^{th}$ rows and the $%
j_{1}^{th},$ $j_{2}^{th};...j_{l}^{th}$ columns is denoted by $\left(
a_{i,j}\right) _{\substack{ i\neq i_{1},i_{2},...i_{k}  \\ j\neq
j_{1},j_{2},...j_{l}}}$.
\end{notation}

\begin{notation}
we denote by $A(k)$ to the $\left( n-2\right) $ square matrix obtained from $%
A$ by removing the rows $\left( n-2\right) $ and $\left( n-1\right) $ and
the columns $\left( n-1\right) $ and $n$ and replacing the column $\left(
n-2\right) $ by the $\left( n-k\right) $, i.e%
\begin{equation}
A(k)=\left( 
\begin{array}{ccccc}
a_{11} & a_{12} & \cdots & a_{1\left( n-3\right) } & a_{1\left( n-k\right) }
\\ 
a_{21} & a_{22} & \cdots & a_{2\left( n-3\right) } & a_{2\left( n-k\right) }
\\ 
\vdots & \vdots & \ddots & \vdots & \vdots \\ 
a_{\left( n-3\right) 1} & a_{\left( n-3\right) 2} & \cdots & a_{\left(
n-3\right) \left( n-3\right) } & a_{\left( n-3\right) \left( n-k\right) } \\ 
a_{n1} & a_{n2} & \cdots & a_{n\left( n-3\right) } & a_{n\left( n-k\right) }%
\end{array}%
\right) ,  \label{(1.2)}
\end{equation}%
also it can be written as a block matrix :\newline
\begin{equation*}
A(k)=\left( 
\begin{array}{cc}
\left( a_{i,j}\right) _{\substack{ i\neq n,n-1,n-2  \\ j\neq n-1,n-2}} & 
\begin{array}{c}
a_{1\left( n-k\right) } \\ 
a_{2\left( n-k\right) } \\ 
\vdots \\ 
a_{\left( n-3\right) \left( n-k\right) } \\ 
a_{n\left( n-k\right) }%
\end{array}%
\end{array}%
\right) \text{.}
\end{equation*}
\end{notation}

\begin{notation}
The $n$ square matrix obtained from $A$ by replacing the row $n$ by the $%
\left( n-l\right) $ one is denoted by $B\left( l\right) $, i.e%
\begin{equation}
B\left( l\right) =\left( 
\begin{array}{ccccccc}
a_{11} & a_{12} & \cdots & a_{1\left( n-3\right) } & a_{1\left( n-2\right) }
& a_{1\left( n-1\right) } & a_{1n} \\ 
a_{21} & a_{22} & \cdots & a_{2\left( n-3\right) } & a_{2\left( n-2\right) }
& a_{2\left( n-1\right) } & a_{2n} \\ 
\vdots & \vdots & \ddots & \vdots & \vdots & \vdots & \vdots \\ 
a_{\left( n-3\right) 1} & a_{\left( n-3\right) 2} & \cdots & a_{\left(
n-3\right) \left( n-3\right) } & a_{\left( n-3\right) \left( n-2\right) } & 
a_{\left( n-3\right) \left( n-1\right) } & a_{\left( n-3\right) n} \\ 
a_{\left( n-2\right) 1} & a_{\left( n-2\right) 2} & \cdots & a_{\left(
n-2\right) \left( n-3\right) } & a_{\left( n-2\right) \left( n-2\right) } & 
a_{\left( n-2\right) \left( n-1\right) } & a_{\left( n-2\right) n} \\ 
a_{\left( n-1\right) 1} & a_{\left( n-1\right) 2} & \cdots & a_{\left(
n-1\right) \left( n-3\right) } & a_{\left( n-1\right) \left( n-2\right) } & 
a_{\left( n-1\right) \left( n-1\right) } & a_{\left( n-1\right) n} \\ 
a_{(n-l)1} & a_{(n-l)2} & \cdots & a_{(n-l)\left( n-3\right) } & 
a_{(n-l)\left( n-2\right) } & a_{(n-l)\left( n-1\right) } & a_{(n-l)n}%
\end{array}%
\right) .  \label{(1.3)}
\end{equation}
\end{notation}

\section{RESULTS}

We need some lemmas:

\begin{lemma}
We have%
\begin{equation}
A(0)=\left( a_{i,j}\right) _{\substack{ i\neq n-1,n-2  \\ j\neq n-1,n-2}}%
,A(1)=\left( a_{i,j}\right) _{\substack{ i\neq n-1,n-2  \\ j\neq n,n-2}}%
,A(2)=\left( a_{i,j}\right) _{\substack{ i\neq n-1,n-2  \\ j\neq n,n-1}},
\label{(2.1)}
\end{equation}%
\begin{equation}
\det \left[ A(k)\right] =0\text{ for all }k=3,...,n-1  \label{(2.2)}
\end{equation}%
and%
\begin{eqnarray}
\det \left[ A(k)\right] &=&a_{n\left( n-k\right) }\det \left[ \left(
a_{i,j}\right) _{\substack{ i\neq n,n-1,n-2  \\ j\neq n,n-1,n-2}}\right] + 
\notag \\
&&\overset{l=n-1}{\underset{l=3}{\tsum }}\left( -1\right) ^{l}\left(
a_{n-l,n-k}\right) \det \left[ \left( a_{i,j}\right) _{\substack{ i\neq
n-1,n-2,n-l  \\ j\neq n,n-1,n-2}}\right]  \label{(2.3)}
\end{eqnarray}
\end{lemma}

\begin{proof}
Formula 
\msihyperref{(2.1)}{(}{)}{(2.1)}
is trivial by replacing in 
\msihyperref{(1.2)}{(}{)}{(1.2)}%
, $k$ by $0$, $1$ and $2$ respectively. Formula 
\msihyperref{(2.2)}{(}{)}{(2.2)}
is also trivial since the columns $\left( n-2\right) $ and $\left(
n-k\right) $ are equal. Formula 
\msihyperref{(2.3)}{(}{)}{(2.3)}
is a simple developping of $\det A(k)$ according to the last column.
\end{proof}

\begin{lemma}
We have%
\begin{equation}
B(0)=A,  \label{(2.4)}
\end{equation}%
\begin{equation}
\det \left[ B(l)\right] =0\text{ for all }l=1,...,n-1  \label{(2.5)}
\end{equation}%
and%
\begin{equation}
\det \left[ B(l)\right] =\overset{n-1}{\underset{k=0}{\tsum }}\left(
-1\right) ^{k}\left( a_{n-l,n-k}\right) \det \left[ \left( a_{i,j}\right) 
_{\substack{ i\neq n  \\ j\neq n-k}}\right] .  \label{(2.6)}
\end{equation}
\end{lemma}

\begin{proof}
formula 
\msihyperref{(2.4)}{(}{)}{(2.4)}
is trivial by replacing in 
\msihyperref{(1.2)}{(}{)}{(1.2)}
$l$ by $0.$ Formula 
\msihyperref{(2.5)}{(}{)}{(2.5)}
is also trivial since the rows $n$ and $\left( n-1\right) $ are equal.
Formula 
\msihyperref{(2.6)}{(}{)}{(2.6)}
is a simple developping of $\det B(l)$.according to the last row.
\end{proof}

\begin{lemma}
IF the determinant of the matrix $\left( a_{i,j}\right) _{\substack{ i\neq
n-1,n  \\ j\neq n-1,n}}$ is null, thus we get the following formula: 
\begin{equation}
\det \left[ \left( a_{i,j}\right) _{\substack{ i\neq n  \\ j\neq n}}\right]
\det \left[ \left( a_{i,j}\right) _{\substack{ i\neq n-1  \\ j\neq n-1}}%
\right] -\det \left[ \left( a_{i,j}\right) _{\substack{ i\neq n  \\ j\neq
n-1 }}\right] \det \left[ \left( a_{i,j}\right) _{\substack{ i\neq n-1  \\ %
j\neq n }}\right] =0  \label{(2.7)}
\end{equation}
\end{lemma}

\begin{proof}
To prove the formula 
\msihyperref{(2.7)}{(}{)}{(2.7)}%
, we let $\det \left[ \left( a_{i,j}\right) _{\substack{ i\neq n-1,n  \\ %
j\neq n-1,n}}\right] =0.$ But $\det \left[ \left( a_{i,j}\right) _{\substack{
i\neq n-1,n  \\ j\neq n-1,n}}\right] =0$ it means the rows of the matrix $%
\left( a_{i,j}\right) _{\substack{ i\neq n-1,n  \\ j\neq n-1,n}}$ \ are
dependent linearly.Consequently, $\exists \left( \lambda _{i}\right)
_{i=1}^{n-2},$ $\exists k=1,...,n-2$\newline
\begin{equation}
\det \left[ \left( a_{i,j}\right) _{\substack{ i\neq 2n-1-s  \\ j\neq 2n-1-r 
}}\right] =\left\vert 
\begin{array}{cccccccc}
a_{11} & a_{12} & a_{13} & .. & a_{1k} & ... & a_{1n-2} & a_{1r} \\ 
a_{21} & a_{22} & a_{23} & ... & a_{2k} & ... & a_{2n-2} & a_{2r} \\ 
a_{31} & a_{32} & a_{33} & ... & a_{3k} & ... & a_{3n-2} & a_{3r} \\ 
\vdots & \vdots & \vdots & \ddots & \vdots & \ddots & \vdots & \vdots \\ 
0 & 0 & 0 & ... & 0 & ... & 0 & \underset{i=1}{\overset{i=n-2}{\tsum }}%
\lambda _{i}a_{ir} \\ 
\vdots & \vdots & \vdots & \ddots & \vdots & \ddots & \vdots & \vdots \\ 
a_{n-2,1} & a_{n-2,2} & a_{n-2,3} & ... & a_{n-2,k} & ... & a_{n-2,n-2} & 
a_{n-2,r} \\ 
a_{s,1} & a_{s,2} & a_{s,3} & ... & a_{s,k} & ... & a_{s,,n-2} & a_{s,r}%
\end{array}%
\right\vert  \label{(2.8)}
\end{equation}%
for $r,s=\left( n-1\right) $ or $\ n..$ And also by calculating the
determinant according to the last column, the formula will be the simple
shape: 
\begin{equation}
\left\{ 
\begin{array}{c}
\det \left[ \left( a_{i,j}\right) _{\substack{ i\neq 2n-1-s  \\ j\neq 2n-1-r 
}}\right] = \\ 
\left( -1\right) ^{n-k+1}\left( \underset{i=1}{\overset{i=n-2}{\tsum }}%
\lambda _{i}a_{ir}\right) \det \left[ \left( a_{i,j}\right) _{\substack{ %
i\neq 2n-1-s,k  \\ j\neq 2n-1-r,r}}\right]%
\end{array}%
\right. \text{ where }r,s=n-1\text{ or }n  \label{(2.9)}
\end{equation}%
By using the formula 
\msihyperref{(2.9)}{(}{)}{(2.9)}
,we get :\newline
\begin{equation*}
\det \left[ \left( a_{i,j}\right) _{\substack{ i\neq n  \\ j\neq n}}\right]
=\left( -1\right) ^{n-k+1}\left( \underset{i=1}{\overset{i=n-2}{\tsum }}%
\lambda _{i}a_{in-1}\right) \det \left[ \left( a_{i,j}\right) _{\substack{ %
i\neq n,k  \\ j\neq n,n-1}}\right] ;
\end{equation*}%
\begin{equation*}
\det \left[ \left( a_{i,j}\right) _{\substack{ i\neq n-1  \\ j\neq n-1}}%
\right] =\left( -1\right) ^{n-k+1}\left( \underset{i=1}{\overset{i=n-2}{%
\tsum }}\lambda _{i}a_{in}\right) \det \left[ \left( a_{i,j}\right) 
_{\substack{ i\neq n-1,k  \\ j\neq n-1,n}}\right] ;
\end{equation*}%
\begin{equation*}
\det \left[ \left( a_{i,j}\right) _{\substack{ i\neq n  \\ j\neq n-1}}\right]
=\left( -1\right) ^{n-k+1}\left( \underset{i=1}{\overset{i=n-2}{\tsum }}%
\lambda _{i}a_{in}\right) \det \left[ \left( a_{i,j}\right) _{\substack{ %
i\neq n,k  \\ j\neq n-1,n}}\right] ;
\end{equation*}%
\begin{equation*}
\det \left[ \left( a_{i,j}\right) _{\substack{ i\neq n-1  \\ j\neq n}}\right]
=\left( -1\right) ^{n-k+1}\left( \underset{i=1}{\overset{i=n-2}{\tsum }}%
\lambda _{i}a_{in-1}\right) \det \left[ \left( a_{i,j}\right) _{\substack{ %
i\neq n-1,k  \\ j\neq n,n-1}}\right] .
\end{equation*}%
Finally, the first member%
\begin{equation*}
\det \left[ \left( a_{i,j}\right) _{\substack{ i\neq n  \\ j\neq n}}\right]
\det \left[ \left( a_{i,j}\right) _{\substack{ i\neq n-1  \\ j\neq n-1}}%
\right] -\det \left[ \left( a_{i,j}\right) _{\substack{ i\neq n  \\ j\neq
n-1 }}\right] \det \left[ \left( a_{i,j}\right) _{\substack{ i\neq n-1  \\ %
j\neq n }}\right]
\end{equation*}%
of the formula 
\msihyperref{(2.7)}{(}{)}{(2.7)}
will be: $\left( \underset{i=1}{\overset{i=n-2}{\tsum }}\lambda
_{i}a_{in-1}\right) \left( \underset{i=1}{\overset{i=n-2}{\tsum }}\lambda
_{i}a_{in}\right) \cdot $\newline
$\left[ \det \left[ \left( a_{i,j}\right) _{\substack{ i\neq n,k  \\ j\neq
n,n-1}}\right] \det \left[ \left( a_{i,j}\right) _{\substack{ i\neq n-1,k 
\\ j\neq n-1,n}}\right] -\det \left[ \left( a_{i,j}\right) _{\substack{ %
i\neq n,k  \\ j\neq n,n-1}}\right] \det \left[ \left( a_{i,j}\right) 
_{\substack{ i\neq n-1,k  \\ j\neq n-1,n}}\right] \right] =0$. And like
this, we have finished the proof of the lemma 3.
\end{proof}

\begin{theorem}
For all square matrix $A=\left( a_{i,j}\right) _{1\leq i,j\leq n}$ where $%
n>2 $\newline
the formula 
\msihyperref{(1.1)}{(}{)}{(1.1)}
is satisfied
\end{theorem}

\begin{proof}
The formula 
\msihyperref{(1.1)}{(}{)}{(1.1)}
can be written as follows:\newline
\begin{equation}
\det \left[ \left( a_{i,j}\right) _{1\leq i,j\leq n}\right] \det \left[
\left( a_{i,j}\right) _{\substack{ i\neq k,l  \\ j\neq k,l}}\right] =\det %
\left[ 
\begin{array}{cc}
\det \left[ \left( a_{i,j}\right) _{\substack{ i\neq l  \\ j\neq l}}\right]
& \det \left[ \left( a_{i,j}\right) _{\substack{ i\neq l  \\ j\neq k}}\right]
\\ 
\det \left[ \left( a_{i,j}\right) _{\substack{ i\neq k  \\ j\neq l}}\right]
& \det \left[ \left( a_{i,j}\right) _{\substack{ i\neq k  \\ j\neq k}}\right]%
\end{array}%
\right] .  \label{(2.10)}
\end{equation}

To prove the formula 
\msihyperref{(1.1)}{(}{)}{(1.1)}%
, it is sufficient to prove the following formula:\newline
for $n>2$ we have:%
\begin{equation}
\left. 
\begin{array}{c}
\det \left[ \left( a_{i,j}\right) _{1\leq i,j\leq n}\right] \det \left[
\left( a_{i,j}\right) _{\substack{ i\neq n-1,n  \\ j\neq n-1,n}}\right] = \\ 
\det \left[ \left( a_{i,j}\right) _{\substack{ i\neq n  \\ j\neq n}}\right]
\det \left[ \left( a_{i,j}\right) _{\substack{ i\neq n-1  \\ j\neq n-1}}%
\right] -\det \left[ \left( a_{i,j}\right) _{\substack{ i\neq n  \\ j\neq
n-1 }}\right] \det \left[ \left( a_{i,j}\right) _{\substack{ i\neq n-1  \\ %
j\neq n }}\right]%
\end{array}%
\right.  \label{(2.11)}
\end{equation}%
Because by pair replacing of row $k$ with row $(n-1)$ and row $l$ with row $%
n $, column $k$ with column $(n-1)$ and column $l$ with column $n$ \ that
does not change the determinant of the matrix $\left( a_{i,j}\right) _{1\leq
i,j\leq n}$. And by applying the formula 
\msihyperref{(2.11)}{(}{)}{(2.11)}
,then we replace some rows and some columns, noting that the number of the
replaced rows is the same number of the replaced columns in order to stable
the determinants, consequently, we get 
\msihyperref{(1.1)}{(}{)}{(1.1)}%
.

To prove the formula 
\msihyperref{(2.11)}{(}{)}{(2.11)}
there are two cases:

The first case: when $\det \left[ \left( a_{i,j}\right) _{\substack{ i\neq
n-1,n  \\ j\neq n-1,n}}\right] =0$, the proof of the formula 
\msihyperref{(2.11)}{(}{)}{(2.11)}
is the same proof of lemma3.

The second case: when $\det \left[ \left( a_{i,j}\right) _{\substack{ i\neq
n-1,n  \\ j\neq n-1,n}}\right] \neq 0$ we prove the formula 
\msihyperref{(2.11)}{(}{)}{(2.11)}
inductively:

For $n=3$, we find that the proof of the formula 
\msihyperref{(2.11)}{(}{)}{(2.11)}
is evident. For $n=4$, we find that the proof of the formula 
\msihyperref{(2.11)}{(}{)}{(2.11)}
is evident with simple calculations.

When $n>4$, we suppose the formula 
\msihyperref{(2.11)}{(}{)}{(2.11)}
is correct for $\left( n-1\right) $ and we prove it for $n.$ In an other
word, we prove that :%
\begin{equation}
P-Q=\det \left[ \left( a_{i,j}\right) _{1\leq i,j\leq n}\right] .\det \left[
\left( a_{i,j}\right) _{\substack{ i\neq n-1,n  \\ j\neq n-1,n}}\right]
\label{(2.12)}
\end{equation}%
where%
\begin{equation}
P=\det \left[ \left( a_{i,j}\right) _{\substack{ i\neq n  \\ j\neq n}}\right]
\det \left[ \left( a_{i,j}\right) _{\substack{ i\neq n-1  \\ j\neq n-1}}%
\right]  \label{(2.13)}
\end{equation}%
and 
\begin{equation}
Q=\det \left[ \left( a_{i,j}\right) _{\substack{ i\neq n  \\ j\neq n-1}}%
\right] \det \left[ \left( a_{i,j}\right) _{\substack{ i\neq n-1  \\ j\neq n 
}}\right] .  \label{(2.14)}
\end{equation}

We apply the formula 
\msihyperref{(2.11)}{(}{)}{(2.11)}
for $\left( n-1\right) $ on the formula 
\msihyperref{(2.13)}{(}{)}{(2.13)}%
\begin{equation}
\left. 
\begin{array}{c}
\det \left[ \left( a_{i,j}\right) _{\substack{ i\neq n  \\ j\neq n}}\right]
\det \left[ \left( a_{i,j}\right) _{\substack{ i\neq n,n-1,n-2  \\ j\neq
n,n-1,n-2}}\right] = \\ 
\det \left[ \left( a_{i,j}\right) _{\substack{ i\neq n,n-1  \\ j\neq n,n-1}}%
\right] \det \left[ \left( a_{i,j}\right) _{\substack{ i\neq n,n-2  \\ j\neq
n,n-2}}\right] -\det \left[ \left( a_{i,j}\right) _{\substack{ i\neq n,n-1 
\\ j\neq n,n-2}}\right] \det \left[ \left( a_{i,j}\right) _{\substack{ i\neq
n,n-2  \\ j\neq n,n-1}}\right] .%
\end{array}%
\right.  \label{(2.15)}
\end{equation}%
\begin{equation}
\left. 
\begin{array}{c}
\det \left[ \left( a_{i,j}\right) _{\substack{ i\neq n-1  \\ j\neq n-1}}%
\right] \det \left[ \left( a_{i,j}\right) _{\substack{ i\neq n,n-1,n-2  \\ %
j\neq n,n-1,n-2}}\right] = \\ 
\det \left[ \left( a_{i,j}\right) _{\substack{ i\neq n,n-1  \\ j\neq n,n-1}}%
\right] \det \left[ \left( a_{i,j}\right) _{\substack{ i\neq n-1,n-2  \\ %
j\neq n-1,n-2}}\right] -\det \left[ \left( a_{i,j}\right) _{\substack{ i\neq
n,n-1  \\ j\neq n-1,n-2}}\right] \det \left[ \left( a_{i,j}\right) 
_{\substack{ i\neq n-1,n-2  \\ j\neq n,n-1}}\right] .%
\end{array}%
\right.  \label{(2.16)}
\end{equation}%
We apply the formula 
\msihyperref{(2.11)}{(}{)}{(2.11)}
for $\left( n-1\right) $ on the formula 
\msihyperref{(2.14)}{(}{)}{(2.14)}%
: 
\begin{equation}
\left. 
\begin{array}{c}
\det \left[ \left( a_{i,j}\right) _{\substack{ i\neq n  \\ j\neq n-1}}\right]
\det \left[ \left( a_{i,j}\right) _{\substack{ i\neq n,n-1,n-2  \\ j\neq
n,n-1,n-2}}\right] = \\ 
\det \left[ \left( a_{i,j}\right) _{\substack{ i\neq n,n-1  \\ j\neq n,n-1}}%
\right] \det \left[ \left( a_{i,j}\right) _{\substack{ i\neq n,n-2  \\ j\neq
n-1,n-2}}\right] -\det \left[ \left( a_{i,j}\right) _{\substack{ i\neq n,n-1 
\\ j\neq n-1,n-2}}\right] \det \left[ \left( a_{i,j}\right) _{\substack{ %
i\neq n,n-2  \\ j\neq n,n-1}}\right] .%
\end{array}%
\right.  \label{(2.17)}
\end{equation}%
\begin{equation}
\left. 
\begin{array}{c}
\det \left[ \left( a_{i,j}\right) _{\substack{ i\neq n-1  \\ j\neq n}}\right]
\det \left[ \left( a_{i,j}\right) _{\substack{ i\neq n,n-1,n-2  \\ j\neq
n,n-1,n-2}}\right] = \\ 
\det \left[ \left( a_{i,j}\right) _{\substack{ i\neq n,n-1  \\ j\neq n,n-1}}%
\right] \det \left[ \left( a_{i,j}\right) _{\substack{ i\neq n-1,n-2  \\ %
j\neq n,n-2}}\right] -\det \left[ \left( a_{i,j}\right) _{\substack{ i\neq
n,n-1  \\ j\neq n,n-2}}\right] \det \left[ \left( a_{i,j}\right) _{\substack{
i\neq n-1,n-2  \\ j\neq n,n-1}}\right] .%
\end{array}%
\right.  \label{(2.18)}
\end{equation}%
Using the formulas 
\msihyperref{(2.15)}{(}{)}{(2.15)}%
-%
\msihyperref{(2.18)}{(}{)}{(2.18)}
on $\left( P-Q\right) $, we find : 
\begin{equation}
\left. 
\begin{array}{c}
\left( P-Q\right) \left( \det \left[ \left( a_{i,j}\right) _{\substack{ %
i\neq n,n-1,n-2  \\ j\neq n,n-1,n-2}}\right] \right) ^{2}= \\ 
\det \left[ \left( a_{i,j}\right) _{\substack{ i\neq n,n-1  \\ j\neq n,n-1}}%
\right] \cdot \left\{ 
\begin{array}{c}
\Gamma \det \left[ \left( a_{i,j}\right) _{\substack{ i\neq n-1,n-2  \\ %
j\neq n-1,n-2}}\right] -\Lambda \det \left[ \left( a_{i,j}\right) 
_{\substack{ i\neq n-1,n-2  \\ j\neq n,n-2}}\right] \\ 
+\digamma \det \left[ \left( a_{i,j}\right) _{\substack{ i\neq n-1,n-2  \\ %
j\neq n,n-1}}\right]%
\end{array}%
\right\}%
\end{array}%
\right.  \label{(2.19)}
\end{equation}%
where $\Gamma =\det \left[ \left( a_{i,j}\right) _{\substack{ i\neq n,n-2 
\\ j\neq n,n-2}}\right] \det \left[ \left( a_{i,j}\right) _{\substack{ i\neq
n,n-1  \\ j\neq n,n-1}}\right] -\det \left[ \left( a_{i,j}\right) 
_{\substack{ i\neq n,n-1  \\ j\neq n,n-2}}\right] \det \left[ \left(
a_{i,j}\right) _{\substack{ i\neq n,n-2  \\ j\neq n,n-1}}\right] $ $\Lambda
=\det \left[ \left( a_{i,j}\right) _{\substack{ i\neq n,n-2  \\ j\neq
n-1,n-2 }}\right] \det \left[ \left( a_{i,j}\right) _{\substack{ i\neq n,n-1 
\\ j\neq n,n-1}}\right] -\det \left[ \left( a_{i,j}\right) _{\substack{ %
i\neq n,n-1  \\ j\neq n-1,n-2}}\right] \det \left[ \left( a_{i,j}\right) 
_{\substack{ i\neq n,n-2  \\ j\neq n,n-1}}\right] $ $\digamma =\det \left[
\left( a_{i,j}\right) _{\substack{ i\neq n,n-2  \\ j\neq n-1,n-2}}\right]
\det \left[ \left( a_{i,j}\right) _{\substack{ i\neq n,n-1  \\ j\neq n,n-2}}%
\right] -\det \left[ \left( a_{i,j}\right) _{\substack{ i\neq n,n-2  \\ %
j\neq n,n-2}}\right] \det \left[ \left( a_{i,j}\right) _{\substack{ i\neq
n,n-1  \\ j\neq n-1,n-2}}\right] $ We apply the formula 
\msihyperref{(2.11)}{(}{)}{(2.11)}
for $\left( n-1\right) $ on $\Gamma ,\Lambda $ and $\digamma ,$ we find the
following : 
\begin{equation}
\Gamma =\det \left[ \left( a_{i,j}\right) _{\substack{ i\neq n  \\ j\neq n}}%
\right] \det \left[ \left( a_{i,j}\right) _{\substack{ i\neq n,n-1,n-2  \\ %
j\neq n,n-1,n-2}}\right]  \label{(2.20)}
\end{equation}%
\begin{equation}
\Lambda =\det \left[ \left( a_{i,j}\right) _{\substack{ i\neq n  \\ j\neq
n-1 }}\right] \det \left[ \left( a_{i,j}\right) _{\substack{ i\neq n,n-1,n-2 
\\ j\neq n,n-1,n-2}}\right]  \label{(2.21)}
\end{equation}%
\begin{equation}
\digamma =\det \left[ \left( a_{i,j}\right) _{\substack{ i\neq n  \\ j\neq
n-2 }}\right] \det \left[ \left( a_{i,j}\right) _{\substack{ i\neq n,n-1,n-2 
\\ j\neq n,n-1,n-2}}\right]  \label{(2.22)}
\end{equation}%
In the formula 
\msihyperref{(2.19), }{(}{)}{(2.19)}%
we replace $\Gamma ,\Lambda $ and $\digamma $ by their equivalents, we find:%
\begin{equation}
\left. 
\begin{array}{c}
\left( P-Q\right) \det \left[ \left( a_{i,j}\right) _{\substack{ i\neq
n,n-1,n-2  \\ j\neq n,n-1,n-2}}\right] = \\ 
\det \left[ \left( a_{i,j}\right) _{\substack{ i\neq n,n-1  \\ i\neq n,n-1}}%
\right] \cdot \left\{ 
\begin{array}{c}
\det \left[ \left( a_{i,j}\right) _{\substack{ i\neq n-1,n-2  \\ j\neq
n-1,n-2 }}\right] \det \left[ \left( a_{i,j}\right) _{\substack{ i\neq n  \\ %
j\neq n}}\right] \\ 
-\det \left[ \left( a_{i,j}\right) _{\substack{ i\neq n-1,n-2  \\ j\neq
n,n-2 }}\right] \det \left[ \left( a_{i,j}\right) _{\substack{ i\neq n  \\ %
j\neq n-1 }}\right] \\ 
+\det \left[ \left( a_{i,j}\right) _{\substack{ i\neq n-1,n-2  \\ j\neq
n,n-1 }}\right] \det \left[ \left( a_{i,j}\right) _{\substack{ i\neq n  \\ %
j\neq n-2 }}\right]%
\end{array}%
\right\} \text{.}%
\end{array}%
\right.  \label{(2.23)}
\end{equation}%
By using the formula 
\msihyperref{(2.1)}{(}{)}{(2.1)}
in lemma 1, we can write the formula 
\msihyperref{(2.23)}{(}{)}{(2.23)}
as follows:%
\begin{equation*}
\left( P-Q\right) \det \left[ \left( a_{i,j}\right) _{\substack{ i\neq
n,n-1,n-2  \\ j\neq n,n-1,n-2}}\right] =\det \left[ \left( a_{i,j}\right) 
_{\substack{ i\neq n,n-1  \\ j\neq n,n-1}}\right] \overset{2}{\underset{k=0}{%
\tsum }}\left( -1\right) ^{k}\det \left[ A(k)\right] \det \left[ \left(
a_{i,j}\right) _{\substack{ i\neq n  \\ i\neq n-k}}\right] .
\end{equation*}%
By using the formula 
\msihyperref{(2.2)}{(}{)}{(2.2)}
in lemma 1, we can write the above formula as follows:%
\begin{equation*}
\left( P-Q\right) \det \left[ \left( a_{i,j}\right) _{\substack{ i\neq
n,n-1,n-2  \\ j\neq n,n-1,n-2}}\right] =\det \left[ \left( a_{i,j}\right) 
_{\substack{ i\neq n,n-1  \\ j\neq n,n-1}}\right] \overset{n-1}{\underset{k=0%
}{\tsum }}\left( -1\right) ^{k}\det \left[ A(k)\right] \det \left[ \left(
a_{i,j}\right) _{\substack{ i\neq n  \\ i\neq n-k}}\right] .
\end{equation*}%
By using the formula 
\msihyperref{(2.3)}{(}{)}{(2.3)}
in lemma 1 and the formula 
\msihyperref{(2.6)}{(}{)}{(2.6)}
in lemma 2, we can write the above formula as follows 
\begin{eqnarray*}
\left( P-Q\right) \det \left[ \left( a_{i,j}\right) _{\substack{ i\neq
n,n-1,n-2  \\ j\neq n,n-1,n-2}}\right] &=&\det \left[ \left( a_{i,j}\right) 
_{\substack{ i\neq n,n-1,n-2  \\ j\neq n,n-1,n-2}}\right] \det \left[ \left(
a_{i,j}\right) _{\substack{ i\neq n,n-1  \\ j\neq n,n-1}}\right] \det \left[
\left( a_{i,j}\right) _{1\leq i,j\leq n}\right] \\
&&+\overset{l=n-1}{\underset{l=3}{\tsum }}\left( -1\right) ^{l}\left\{ \det
B(l)\right\} \det \left[ \left( a_{i,j}\right) _{\substack{ i\neq
n-1,n-2,n-l  \\ j\neq n,n-1,n-2}}\right] .
\end{eqnarray*}%
\newline
And finally, by using the formula 
\msihyperref{(2.5)}{(}{)}{(2.5)}
in lemma 2, we find that the above formula will be formula 
\msihyperref{(2.12)}{(}{)}{(2.12)}%
.
\end{proof}

\end{document}